\documentclass[a4paper,11pt]{article}

\usepackage[utf8]{inputenc}
\usepackage[T1]{fontenc}
\usepackage{lmodern}
\usepackage{a4wide}

\usepackage[english]{babel}
\usepackage[tbtags]{amsmath}
\usepackage{amssymb}
\usepackage{amsfonts}
\usepackage{amsthm}
\usepackage{array}

\newcommand{\m}[1]{\mathbb{#1}}
\newcommand{\gh}[1]{\mathfrak{#1}}
\newcommand{\q}[1]{\mathcal{#1}}
\newcommand{\g}[1]{\ensuremath{\mathbf{#1}}}

\newcommand{\e}{\varepsilon}

\newcommand{\tendf}{\rightharpoonup}

\renewcommand{\Re}{\mathrm{Re}}
\renewcommand{\Im}{\mathrm{Im}}

\DeclareMathOperator{\Span}{\mathrm{Span}}

\theoremstyle{plain}
\newtheorem{thm}{Theorem}
\newtheorem*{thm*}{Theorem}
\newtheorem{prop}{Proposition}

\newtheorem{lem}{Lemma}

\theoremstyle{definition}
\newtheorem{defi}{Definition}

\theoremstyle{remark}
\newtheorem{nb}{Remark}
\newtheorem*{claim}{Claim}
\makeatletter
\def\blfootnote{\xdef\@thefnmark{}\@footnotetext}
\makeatother

\title{Construction of multi-soliton solutions \\ for the $L^2$-supercritical gKdV and NLS equations
\footnote{
This research was supported in part by the Agence Nationale de la Recherche
(ANR ONDENONLIN).}}

\author{Raphaël Côte$^{(1)}$ \and Yvan Martel$^{(2)}$ \and Frank Merle$^{(3)}$}
\date{\small (1)  CNRS and CMLS, UMR 7640, Ecole Polytechnique, 91128 Palaiseau, France \\ 
(2) Mathématiques, UMR 8100,
Univ.  Versailles-Saint-Quentin-en-Y., 78035 Versailles, France,\\
Institut Universitaire de France and CNRS\\  
(3) Mathématiques, UMR 8088, Univ. Cergy-Pontoise, 95302 Cergy-Pontoise, France  and IHES
}

\begin{document}

\maketitle

\begin{abstract}
Multi-soliton solutions, i.e. solutions behaving as the sum of $N$ given solitons as $t\to +\infty$, were constructed in previous works for the $L^2$ critical and subcritical (NLS) and (gKdV) equations
	(see \cite{Mer90}, \cite{Mar05} and \cite{MM06}).
	In this paper, we extend the  construction of  multi-soliton solutions to the $L^2$ supercritical case both
	for  (gKdV) and  (NLS) equations, using a topological argument to control the direction of instability.
\end{abstract}

\section{Introduction}

\subsection{The generalized KdV equation}

\noindent We consider the generalized Korteweg-de Vries equations~:
\begin{equation}
\label{gkdv} \tag{gKdV}
u_t + (u_{xx} + u^p )_x =0, \quad (t,x) \in \m R \times \m R, 
\end{equation}
where $p \ge 2$ is an integer. See Section \ref{sec:gkdv} for more general nonlinearities.

Recall that the Cauchy problem for  (\ref{gkdv}) in the energy space $H^1$  has been solved by  Kenig, Ponce and Vega \cite{KPV93}~: 
for all $u_0 \in H^1(\m R)$, there exist $T=T(\| u_0 \|_{H^1}) >0$ and a solution $u \in \q C([0,T],
H^1(\m R))$ to (\ref{gkdv}) satisfying $u(0) = u_0$, unique in some sense. 
Moreover, if $T_1$ denotes the maximal time
of existence for $u$, then either $T_1=+\infty$ (global solution) or $T_1 <
\infty $ and then $\| u(t) \|_{H^1} \to \infty$ as $t \uparrow T_1$ (blow-up solution).

For such solutions, the mass and energy are conserved~:
\begin{gather}
 \displaystyle \int u^2(t) =  \int u^2(0) \qquad (L^2 \text{ mass}),  \\
 \displaystyle E(u(t)) = \frac 1 2\int u_x^2(t) - \frac 1 {p+1} \int u^{p+1}(t) = E(u(0)) \qquad (\text{energy}).
\end{gather}

\medskip

Now, we define $Q \in H^1$, $Q>0$ the unique solution  (up to translations) to
\[ Q_{xx} + Q^p = Q, \quad \text{ i.e. } \quad Q(x) =
\left( \frac{p+1}{2 \cosh^2(\frac{p-1} 2 x)} \right)^{\frac 1 {p-1}}. \]
Let $Q_{c_0}(x) = c_0^{\frac{1}{p-1}} Q(\sqrt{c_0} x)$ and let
\[ R_{c_0,x_0}(t,x) = c_0^{\frac{1}{p-1}} Q(\sqrt c (x-c_0t -x_0)) \]
be the family of soliton solution of the (gKdV) equation.

It is well-known that the stability properties of a soliton solution depend on the sign of
${\frac d{dc} \int Q_c^2}_{|c=c_0}$. Since $\int Q_c^2 = c^{\frac {5-p}{p-1}} \int Q^2$, we distinguish the following three cases:
\begin{itemize}
	\item For $p <5$ ($L^2$ subcritical case), solitons are stable and asymptotically stable in $H^1$ in some suitable sense~: see Cazenave and Lions \cite{CL82}, Weinstein 
\cite{Wei86} , Grillakis, Shatah and Straus \cite{GSS87}, for orbital stability, and Pego and Weintein \cite{PW94}, Martel and Merle \cite{MM01b} for asymptotic stability. 
	\item In the $L^2$ critical  case, i.e. $p=5$, solitons are unstable, and blow up occur for a large class of solutions initially arbitrarily close to a soliton, see Martel and Merle \cite{MM01a}, \cite{MM02a}.
	\item In the case $p>5$ ($L^2$ supercritical case), solitons are unstable (see Grillakis, Shatah and Straus \cite{GSS87} and Bona, Souganidis and Strauss \cite{BSS87}).
\end{itemize}

Now, we focus on  multi-soliton solutions. Given $2N$ parameters defining $N$ solitons with different speeds,
\begin{equation} \label{parameters}
	0 < c_1 < \ldots < c_N, \quad x_1, \ldots x_N \in \m R,
\end{equation}
we call multi-soliton a solution $u(t)$ to \eqref{gkdv} such that
\begin{equation} \label{problem}
	\left\| u(t) - \sum_{j=1}^N R_{c_j,x_j}(t) \right\|_{H^1} \to 0 \quad \text{as} \quad t \to + \infty.
\end{equation}
Let us recall known results on multi-solitons:
\begin{itemize}
	\item For $p=2$ and $3$ (KdV and mKdV), multi-solitons are well-known to exist for any set of parameters \eqref{parameters}, as a consequence of the inverse scattering method. Moreover, these special explicit solutions describe the elastic collision of the solitons (see e.g. Miura \cite{Miu76}). 
	\item In the $L^2$-subcritical and critical cases, i.e. for \eqref{gkdv} with $p \le 5$ (or for some more general nonlinearities under the stability assumption 
	${\frac d{dc} \int Q_c^2}_{|c=c_j}> 0$ for all $j$), Martel \cite{Mar05} constructed multi-solitons for any set of parameters \eqref{parameters}.
	 The proof of this result follows the strategy of Merle \cite{Mer90} (compactness argument) and relies on monotonicity properties developed in \cite{MM01b} (see also \cite{MMT02}).
	Recall that Martel, Merle and Tsai \cite{MMT02} proved stability and asymptotic stability of a sum of $N$ solitons for large time for the subcritical case.
A refined version of the stability result of \cite{MMT02} shows that for a given  set of parameters, there exists a \emph{unique} multi-soliton soliton satisfying \eqref{problem}, see Theorem~1 in \cite{Mar05}.
\end{itemize}

In the present paper, we extend the multi-soliton existence result to the $L^2$-supercritical case, i.e. in a situation where solitons are known to be unstable.

\begin{thm}[Existence of multi-solitons for $L^2$-supercritical \eqref{gkdv}] \label{nsoliton}
Let $p>5$. Let $0<c_1<\ldots < c_N$ and $x_1, \ldots, x_N \in \m R$. There exist $T_0 \in \m R$, $C,\sigma_0>0,$ and a solution $u \in \q C([T_0,\infty), H^1)$ to \eqref{gkdv} such that
\[ \forall t \in [T_0,\infty), \qquad \left\| u(t) - \sum_{j=1}^N R_{c_j,x_j}(t) \right\|_{H^1} \le C e^{-\sigma_0^{3/2} t}. \]
\end{thm}

\begin{nb}
	As in the subcritical case, the  proof of Theorem \ref{nsoliton} is based on a compactness argument and on some large time uniform estimates, however, it also involves	an additionnal topological argument to control an instable direction of the linearized operator around each $Q_{c_j}$.
	The proof relies decisively on the introduction of $L^2$ eigenfunctions $Y^\pm$ of the linearized operator, contructed by Pego and Weinstein \cite{PW92} by ODE techniques. Note that in \cite{PW92},  the existence of such eigenfunctions for $Q_{c_0}$ is proved to be equivalent to ${\frac d{dc}\int Q_c^2}_{|c=c_0}<0$.
	
	It is possible that other methods of contruction work for some range of parameters $0<c_1<\ldots<c_N$, but due to the instable directions,  the use of such a topological argument is probably necessary to  treat the general case \eqref{parameters}.

	Finally, note that the solution $u(t)$ of Theorem \ref{nsoliton} belongs to $H^s$, and that the convergence
	to $\sum_{j=1}^N R_{c_j,x_j}(t)$ holds in $H^s$, for any $s\geq 1$ (see	Proposition 5 of \cite{Mar05}).
\end{nb}

We refer to Section \ref{sec:gkdv} for a similar existence result for (gKdV) equations with general nonlinearities.

\subsection{The non linear Schrödinger equations}

Now we turn to the case of the non linear Schrödinger equations~:
\begin{equation}
\label{nls} \tag{NLS}
i u_t + \Delta u + |u|^{p-1} u =0, \quad (t,x) \in \m R \times \m R^d, \quad u(t,x) \in \m C,
\end{equation}
where $p >1$, for any space dimension $d\geq 1$. Concerning the local well-posedness of the Cauchy problem in $H^1$, we refer to Ginibre and Velo \cite{GV79}. Recall that $H^1$ solutions satisfy the conservation laws
$$
\int  |u|^2(t)=\int |u_0|^2, \quad \Im \int (\bar u \nabla u )(t)=\Im \int \bar u_0 \nabla u_0,
$$
$$
\frac 12 \int |\nabla u(t)|^2 -\frac 1{p+1} \int |u|^{p+1}(t)
=\frac 12 \int |\nabla u_0(t)|^2 -\frac 1{p+1} \int |u_0|^{p+1}.
$$
Consider the radial positive solution $Q \in H^1(\m R^d)$ to
\begin{equation}\label{Qnls}
	\Delta Q + Q^p  = Q,
\end{equation}
which is the unique positive solution of this equation up to translations. 
We refer to \cite{GNN81}, \cite{BL83} and \cite{Kwo89} for classical existence and uniqueness results on equation \eqref{Qnls}.
Given $v_0,x_0 \in \m R^d$, $\gamma_0 \in \m R$ and $c_0 >0$, the function
\[ R_{c_0,\gamma_0,v_0,x_0}(t,x) = c_0^{\frac{1}{p-1}} Q(\sqrt{c_0} (x-v_0t -x_0)) e^{i (\frac 1 2 v_0\cdot x - \frac{1}{4}\| v_0 \|^2 t+ c_0 t + \gamma_0)} \]
is a soliton solution to \eqref{nls}, moving on the line $x=x_0+v_0t$.

\medskip

We recall the following classical results (for any $d\geq 1$):
\begin{itemize}
\item For $1<p < 1+4/d$, ($L^2$ subcritical case) Cazenave and Lions \cite{CL82} proved that solitons are orbitally stable in $H^1$.
Multi-solitons (defined in a similar way as for  (gKdV)) were constructed in this setting by Martel and Merle \cite{MM06}.

\item In the $L^2$ critical case, $p= 1+4/d$, solitons are unstable, however multi-solitons were constructed by Merle \cite{Mer90}, as a consequence of the construction of special solutions of (NLS) blowing up in finite time at $N$ prescribed points.

\item For $p \in ( 1+ \frac{4}{d}, \frac{d+2}{d-2})$ (for $d=1,2$,  $p > 1+ \frac{4}{d}$)~: solitons are unstable (see \cite{GSS87}).
Recall that $p=\frac{d+2}{d-2}$ corresponds to the critical $\dot H^1$ case.
\end{itemize}

We claim  the following analogue of Theorem  \ref{nsoliton} in the context of the $L^2$ supercritical (NLS) equation.

\begin{thm}[Multi-solitons for $L^2$ supercritical \eqref{nls}] \label{nsolitonnls}
Let $p \in ( 1+ \frac{4}{d}, \frac{d+2}{d-2} )$ ($p > 1+ \frac{4}{d}$ for $d=1,2$). Let $c_1,\ldots, c_N >0$, $\gamma_1, \ldots, \gamma_N \in \m R$, $x_1, \ldots, x_N \in \m R^d$, and $v_1, \ldots, v_N \in \m R^d$ be such that
\[ \forall k \ne j, \quad  v_k \ne v_j. \]
Then there exist $T_0 \in \m R$, $C,\sigma_0>0,$ and a solution $u \in \q C([T_0,\infty), H^1)$ to \eqref{nls} such that
\[ \forall t \in [T_0,\infty), \qquad \left\| u(t) - \sum_{j=1}^N R_{c_j,\gamma_j, v_j,x_j}(t) \right\|_{H^1} \le C e^{-\sigma_0^{3/2} t}. \]
\end{thm}

\begin{nb} 
	The condition on $p$ means that the problem is $L^2$ supercritical but $\dot H^1$ subcritical (for $d\geq 3$).
	In the present paper, we do not treat the $\dot H^1$ critical case -- recall that solitons have then only algebraic decay.
	
	The proof of Theorem \ref{nsolitonnls} is completely similar to the one of Theorem \ref{nsoliton}, see Section \ref{sec:nls}. Note that similarly to the (gKdV) case, we will need eigenfunctions for the linearized operator around $Q$. To obtain these objects for the (NLS) case, we refer to Weinstein \cite{Wei85}, Grillakis \cite{Gri88} and Schlag \cite{Schla06}.
\end{nb}

\bigskip

In Section \ref{sec:out}, we present an outline of the proof of Theorem \ref{nsoliton}.
A complete proof of Theorem \ref{nsoliton} is given in Section \ref{sec:2}. 
Next,  extensions of this result to (gKdV) equations with general nonlinearities are presented without proof in Section \ref{sec:gkdv}.
Finally, a sketch of the proof of Theorem \ref{nsolitonnls} is given in Section \ref{sec:nls}. In the Appendix, we gather the proof of two technical lemmas.

\subsection{Outline of proof of Theorem \ref{nsoliton}}\label{sec:out}

For simplicity, we consider only positive solitons and pure power nonlinearities for (gKdV). The proof follows a similar initial strategy as in the works of Merle \cite{Mer90} or Martel \cite{Mar05}. 

We consider a sequence $S_n \to + \infty$ and we set
\[ R_j(t,x) = R_{c_j,x_j}(t,x), \q\quad R(t,x) = \sum_{j=1}^N R_j(t,x). \]

In the subcritical case (\cite{Mar05} and \cite{MM06}), one considers the sequence $(u_n)$ of solutions to \eqref{gkdv} such that $u_n(S_n) = R(S_n)$.
The goal is then to obtain backwards uniform estimates on $u_n(t) - R(t)$ on some time interval $t \in [T_0,S_n]$, where $T_0$ does not depend on $n$. From these estimates, one can construct the multi-soliton soliton by compactness arguments.
To obtain the uniform estimates, one uses monotonicity properties of local conservation laws and   coercivity property of the Hessian of the energy around a soliton~:
\[ Lv = - v_{xx} - pQ^{p-1} v + v. \]
Indeed, in the subcritical case, it is well-known (see \cite{Wei86}) that $(L v, v) \ge \lambda \| v \|_{H^1}^2$  ($\lambda >0$) provided that  $(v,Q) = (v,Q_x) =0$. These two directions are then controlled by modulation with respect to scaling and translation.

In the supercritical case, one cannot obtain uniform estimates by the same way, since the previous property of $L$ fails. It is known that
$(L \cdot, \cdot)$ is  positive definite up to the directions $Q^{\frac{p+1}{2}}$ and $Q_x$~; the direction $Q_x$ can still be handled using modulation in the translation parameter, but the even direction $Q^{\frac{p+1}{2}}$ cannot be controled  by the scaling parameter as for the subcritical case (this is of course related to the instable nature of the soliton).

At this point, we need the $L^2$ eigenfunctions $Z^\pm$ of the operator $L \partial_x$~:
\[ L(Z^\pm_x) = \pm e_0 Z^\pm, \quad e_0 >0. \]
constructed by Pego and Weinstein \cite{PW92}.
Following  Duyckaerts and Merle \cite{DM07},  we  prove that $(L \cdot, \cdot)$ is positive definite up to the directions $Z^\pm$ and $Q_x$ (see Lemma  \ref{ycoer} in the present paper).
The direction $Z^-$ being in some sense a stable direction, it does not create any  difficulty.
For the instable direction $Z^+$, we do need an extra parameter, which cannot be controlled by a scaling argument.
Thus, instead of considering the final data $u_n(S_n) = R(S_n)$, as in \cite{Mar05}, we look at solutions to \eqref{gkdv} with final data~:
\[ u_n(S_n) = R(S_n) + \sum_{j,\pm} \gh b_{j,n}^\pm Z^\pm_j, \quad \text{where} \quad Z_j(t,x) = c^{\frac{1}{p-1}} Z^\pm(\sqrt{c_j}(x-c_jt - x_j)), \]
and $\gh b_n= (\gh b^\pm_{j,n})_{j=1,\ldots N;\pm}$ belongs to some small neighborhood of $0$ in $\m R^{2N}$. A topological argument then allows us to select, for all $n$,  $\gh b_n$ so that, for the corresponding solution $u_n$, we obtain a uniform control on $\| u_n(t) - R(t) \|_{H^1}$ on some interval $[T_0,S_n]$.

\section{Proof of Theorem \ref{nsoliton}}\label{sec:2}

\subsection{Preliminary results}

Consider the operator 
\[ Lv = - v_{xx} - pQ^{p-1} v+v. \]
For $p>5$, it is known from the work of Pego and Weinstein \cite{PW92} that the operator $\partial_x L$ has two eigenfunctions $Y^+$ and $Y^-$ (related by $Y^-(x) = Y^+(-x)$) such that
\[ (L Y^\pm)_x = \pm e_0 Y^\pm, \text{ where } e_0 >0. \]
In contrast with the (NLS) case (see references in section \ref{sec:nls}), the existence of $Y^\pm$ is not obtained by variational arguments, but by sharp  ODE techniques.
Note that  \cite{PW92} provides a complete description of the spectrum of $\partial_x L$ in $L^2$ for any $p>1$~; in particular, the existence of such eigenfunctions related to $\pm e_0$ with $e_0>0$ is proved to be equivalent to super criticality (i.e. $p>5$ in the present case).

\medskip

Next, we observe that $Z^\pm = LY^\pm$ are eigenfunctions of
$L\partial_x$ (adjoint to $- \partial_x L$). Indeed,
\[ L(Z^\pm_x) = \pm e_0 Z^\pm. \]
The functions $Z^\pm$ are normalized so that $\| Z^\pm \|_{L^2} =1$. Moreover, we recall from \cite{PW92} (standard ODE arguments) that $Z^\pm, Y^\pm \in \q S(\m R)$ and have exponential decay, along with their derivatives. Let $\eta_0>0$ such that
\[ \forall x \in \m R, \qquad |Z^+(x)| + |Z^-(x)| + |Z^+_{x}(x)| + |Z^-_x(x)| \le C e^{-\eta_0 |x|}. \]

Following \cite{DM07} (concerning the (NLS) case), we claim    the following coercivity property of $L$ (for $f,g\in L^2$, $(f,g)=\int fg$ denotes the scalar product in $L^2$).

\begin{lem} \label{ycoer} There exist $\lambda>0$ such that
\[ \forall v \in H^1, \quad (Lv,v) \ge \lambda \| v \|_{H^1}^2 -
\frac{1}{\lambda} \left( (v,Z^+)^2 + (v,Z^-)^2 + (v,Q_x)^2
\right). \]
\end{lem}

\begin{proof} The proof is completely  similar to the one of \cite[Lemma 5.2]{DM07}. It is given here for the reader's convenience.
	
First we recall  the following well-known result.

\medskip

\begin{claim} There exists $\nu>0$ such that 
\begin{equation}\label{cl5}
\forall v \in H^1, \quad (Lv,v) \ge \nu \| v \|_{H^1}^2 -
\frac{1}{\nu} \left( (v,Q^{\frac{p+1}{2}})^2 + (v,Q_x)^2
\right). 
\end{equation}
\end{claim}
Indeed, $Q_x$ and $Q^{\frac{p+1}{2}}$ are two eigenfunctions for
$L$, namely
\[ LQ_x =0 \quad \text{ and } \quad LQ^{\frac{p+1}{2}} = \mu_0 Q^{\frac{p+1}{2}} \text{, where } \mu_0 = - (p^2+1). \]
The claim then follows from Sturm-Liouville theory.

\medskip

To prove the Lemma, it suffices to show that 
\begin{equation}\label{red}
 \text{if } (v,Z^+) = (v,Z^-) = (v,Q_x) =0 \text{ then } (Lv,v) \ge
\lambda \| v \|_{H^1}^2. 
\end{equation}
Let $v$ satisfy the orthogonality conditions in \eqref{red} and 
decompose the functions $v$, $Y^{\pm}$ orthogonaly in $\Span(Q_x, Q^{\frac{p+1}{2}})^\perp$ and $\Span(Q_x, Q^{\frac{p+1}{2}})$ 
\[ v = w + \alpha Q^{\frac{p+1}{2}}, \quad Y^+ = y^+ + \beta Q_x + \gamma
Q^{\frac{p+1}{2}}, \quad Y^- = y^- + \delta Q_x + \eta
Q^{\frac{p+1}{2}}. \]
By symmetry and uniqueness of the orthogonal decomposition, note that $\delta =
- \beta$, $\eta = \gamma$ and $y^+(-x) = y^-(x)$.

First, we claim that the functions $y^+$, $y^-$ are linearly independent. 
Indeed, decompose into even and odd parts
\[ y^+ = y^{e} + y^{o}, \quad y^{-} = y^{e} - y^{o}. \]
Let us prove that $y^e \ne 0$ and $y^o \ne 0$~; we observe from $(L Y^+)_x = e_0 Y^+$ that
\[ (Ly^e)_x = e_0 (y^o + \beta Q_x) - \mu_0 \gamma (Q^{\frac {p+1}2})_x, \quad (Ly^o)_x = e_0
  (y^e + \gamma Q^{\frac{p+1}{2}} ). \]
If $y^o =0$, then $y^e = 0$ and $\gamma =0$, hence $\beta
=0$, and thus $Y^+ = Y^- =0$, which is a contradiction. 
Now, if we assume $y^e=0$, by $(Ly^o)_x = e_0
  (y^e + \gamma Q^{\frac{p+1}{2}} )$ and $\int Q^{\frac {p+1}2}\neq 0$, we obtain $\gamma=0$.
Thus, from $0=(Ly^e)_x = e_0 (y^o + \beta Q_x)$, we get $y^o=0$ and $\beta=0$, so that $Y^+=Y^-=0$, a contradiction.
From the property $y^e \ne 0$ and $y^o \ne 0$,  one
deduces that $a y^+ + by^- =0$ implies $a=b=0$, hence $y^+$ and $y^-$ are linearly independent.

We now go back to the proof of coercivity. Note that
\[ (LY^\pm, Y^\pm) = \pm e_0^{-1} (LY^\pm, (LY^\pm)_x) =0. \]
We compute 
\begin{align*}
0 & = (v, Z^+) = (v, LY^+) = (Lv, Y^+) = (Lw, y^+) + \alpha \mu_0
\gamma \| Q^{\frac{p+1}{2}} \|_{L^2}^2, \\
0 & = (v,Z^-)  = (v, LY^-) = (Lv, Y^-) = (Lw, y^-) + \alpha \mu_0
\gamma \| Q^{\frac{p+1}{2}} \|_{L^2}^2, \\
0 & = (LY^+, Y^+) = (Ly^+, y^+) + \gamma^2 \mu_0 \| Q^{\frac{p+1}{2}}
\|_{L^2}^2, \\ 
0 & = (LY^-, Y^-) = (Ly^-, y^-) + \gamma^2 \mu_0 \| Q^{\frac{p+1}{2}}
\|_{L^2}^2. 
\end{align*}
Hence
\begin{equation} \label{coerdemo}
(Lv,v) = (Lw,w) + \mu_0 \alpha^2 \| Q^{\frac{p+1}{2}} \|_{L^2}^2 = 
(Lw,w) - \frac{(Lw,y^+)(Lw,y^-)}{\sqrt{(Ly^+,y^+)}
  \sqrt{(Lz^-,z^-)}}.
\end{equation} 

Consider 
\[ a = \sup_{\omega \in \Span(y^+,y^-) \setminus \{ 0 \}}
\left|\frac{(L\omega,y^+)}{\sqrt{(L\omega,\omega) (Ly^+,y^+)}} \cdot
\frac{(L\omega,y^-)}{\sqrt{(L\omega,\omega) (Ly^-,y^-)}} \right| . \]
Recall $(L \cdot, \cdot)$ is  positive definite on $\Span(Q_x,
Q^{\frac{p+1}{2}})^\perp$~; applying Cauchy-Schwarz 
inequality to each of the two terms of the product above, we find $a \le 1$. Furthermore,
if $a=1$, there exists $\omega \ne 0$
such that these two   Cauchy-Schwarz inequalities are actually equalities, but this is not possible since $y^+$ and $y^-$ are independent.

Therefore, we have proved that $a <1$. By decomposition on  $\Span(Q_x, Q^{\frac{p+1}{2}})^\perp$, we also obtain for all 
$\omega \in \Span(Q_x, Q^{\frac{p+1}{2}})^\perp$,
$$
\left| \frac{(L\omega ,y^+)(L\omega,y^-)}{\sqrt{(Ly^+,y^+)}
  \sqrt{(Lz^-,z^-)}}	\right| \leq a (L\omega, \omega).
$$
Hence, by \eqref{coerdemo} and next   \eqref{cl5},
\[  (Lv,v) \ge (1-a) (Lw,w)\ge \nu (1-a) \| w \|_{H^1}^2 >0. \]
Thus, for $C = \max ( \frac{4}{\nu (1-a)}, \frac{4}{(1-a) |\mu_0|}  \| Q^{\frac{p+1}{2}} \|_{L^2}^{-2} \| Q^{\frac{p+1}{2}} \|_{H^1}^2)$ we get 
\begin{align*}
 C (Lv,v) & \ge C (1-a) (Lw,w) \ge C \frac{1-a}{2} (Lw,w) + C \frac{1-a}{2} |\mu_0| \alpha^2 \| Q^{\frac{p+1}{2}} \|_{L^2}^2 \\
& \ge 2 \| w \|_{H^1}^2 + 2 \alpha^2 \| Q^{\frac{p+1}{2}} \|_{H^1}^2 
 \ge \| w + \alpha Q^{\frac{p+1}{2}} \|_{H^1}^2  = \| v \|_{H^1}^2. \qedhere
\end{align*}
\end{proof}

\subsection{Main Proposition and proof of Theorem \ref{nsoliton}}

We denote
\begin{equation}\label{defR}\begin{split}
&	R_j(t,x) = R_{c_j,x_j}(t,x)=c_j^{\frac 1 {p-1}} Q(\sqrt{c_j}(x-c_j t -x_j)),\quad 
\quad R(t,x) = \sum_{j=1}^N R_j(t,x), \\
&	Z^\pm_j(t,x) = c_j^{\frac{1}{p-1}} Z^\pm(\sqrt{c_j}(x-c_jt - x_j)). 
\end{split}\end{equation}
Let $S_n \to \infty$ be a increasing sequence of time, $\gh b_n =(\gh b_{j,n}^\pm)_{j,\pm} \in \m R^{2N}$ be a sequence of parameters to be determined, and let $u_n$ be the solution to
\begin{equation} 
\label{gkdvSn}
\left\{ \begin{array}{>{\displaystyle}l}
{u_n}_t + ({u_n}_{xx} + u_n^{p})_x =0, \vphantom{\sum_{j \in \{ 1,\ldots, N\}, \pm}} \\
u_n(S_n) = R(S_n) + \sum_{j \in \{ 1,\ldots, N\}, \pm} \gh b_{j,n}^\pm Z_j^\pm(S_n).
\end{array} \right.
\end{equation}

Let  
\begin{equation} \label{sigma0}
\sigma_0 = \frac 1 4 \min \left\{ \eta_0, e_0^{2/3}c_1, c_1,c_2-c_1, \ldots, c_N-c_{N-1} \right\}.
\end{equation}

\begin{prop} \label{vndecay}
There exist $n_0\geq 0$,  $T_0>0$ and $C>0$  (independent of $n$) such that the following holds. 
For each $n\geq n_0$, there exists $\gh b_n = (\gh b_{j,n}^\pm)_{j,\pm} \in \m R^{2N}$ with 
\[ \left( \sum_{j,\pm} {\gh b^\pm_{j,n}}^2 \right)^{1/2} \leq e^{-\sigma_0^{3/2} S_n}, \] 
and such that the solution $u_n$ to \eqref{gkdvSn} is defined on the interval $[T_0, S_n]$, and satisfies
\[ \forall t \in [T_0, S_n], \quad \left\| u_n(t) - R(t) \right\|_{H^1} \le C e^{-\sigma_0^{3/2} t}. \]
\end{prop}

Assuming this Proposition, we now deduce the proof of Theorem \ref{nsoliton}. The proof of Proposition \ref{vndecay} is postponed to Section \ref{sec:2.3}.

\begin{proof}[Proof of Theorem \ref{nsoliton} assuming Proposition \ref{vndecay}]
It follows closely the proof of Theorem 1 in  \cite{Mar05}. We may assume $n_0=0$ in Proposition \ref{vndecay} without loss of generality.

\medskip

\noindent\emph{Step 1 : Compactness argument.} From Proposition
\ref{vndecay}, there exists a sequence $u_n(t)$ of solutions to (\ref{gkdv}),
defined on $[T_0, S_n]$ and $C_0, \sigma_0 >0$ such that the following uniform estimates hold~:
\begin{equation}\label{truc}
	\forall n \in \m N, \ \forall t \in [T_0,S_n], \qquad \| u_n (t) - R_n(t) \|_{H^1} \le C_0 e^{-\sigma_0^{3/2} t}.
\end{equation}

We claim the following compactness result on the sequence $u_n(T_0)$.
\begin{claim}
\[ \lim_{A \to \infty} \sup_{n \in \m N} \int_{|x| \ge A}  u_n^2(T_0,x) dx
= 0. \]
\end{claim}

\begin{proof}
Let $\e >0$,  $T(\e)\geq T_0$ be such that $C_0 e^{-\sigma_0^{3/2} T(\e)} \le
\sqrt{\e}$ and $n$ large enough so that $S_n\geq T(\e)$. Then
\[ \int (u_n(T(\e)) - R(T(\e))^2 \le \e. \]
Let $A(\e)$ be such that $\int_{|x| \ge A(\e)}
R(T(\e))^2(x) dx \le \e$ ; we get
\[ \int_{|x| \ge A(\e)} u_n^2(T(\e),x) dx \le  4 \e. \]
Let $g(x) \in C^3$ be such that  $g(x) =0$ if $x \le 0$, $g(x)
=1 $ if $x \ge 2$, and furthermore $0 \le g'(x) \le 1$,  $0 \le
g'''(x) \le 1$.

Recall that for $f(x) \in C^3$, we have (Kato's identity \cite{Kat83})
\[\frac{d}{dt} \int u_n^2 f =  - 3 \int ({u_n}_x)^2 f_x + \int u_n^2 f_{xxx} +
\frac{2p}{p+1} \int u_n^{p+1} f_x. \]
For $C(\e)>1$ to be determined later, we thus have :
\begin{multline*}
\frac{d}{dt} \int u_n^2(t,x) g \left( \frac{x-A(\e)}{C(\e)} \right)
=  - \frac{3}{C(\e)} \int ({u_n}_x)^2 g' \left( \frac{x-A(\e)}{C(\e)} \right) \\
+ \frac{1}{C(\e)^3} \int u_n^2 g''' \left( \frac{x-A(\e)}{C(\e)}
\right) + \frac{2p}{(p+1)C(\e)} \int u_n^{p+1} g' \left(
\frac{x-A(\e)}{C(\e)} \right).
\end{multline*}
For $t \ge T_0\geq 0$, $u_n$ satisfies $\| u_n(t) \|_{H^1} \le C_0 + \sum_{j=1}^N \| Q_{c_j} \|_{H^1} \le C^0$, so that~:
\begin{align*}
\lefteqn{ \left| \frac{d}{dt} \int u_n^2(t,x)
g \left( \frac{x-A(\e)}{C(\e)} \right)
\right| } \\
& & \qquad \qquad & \le \frac{1}{C(\e)} \left( 3
\int {u_n}_x^2(t) + \int u_n^2(t) +
\frac{2p}{p+1} \| u_n \|_{L^\infty}^{p-1} \int u_n^2(t) \right) \\
& & & \le  \frac{1}{C(\e)} \left( 3{C^0}^2 +  \frac{2p}{p+1}
2^{(p-1)/2} {C^0}^{p+1} \right).
\end{align*}
Now choose $C(\e)= \max \left\{ 1, \frac{T(\e)-T_0}{\e} \left(
3{C^0}^2 + \frac{2p}{p+1} 2^{(p-1)/2} {C^0}^{p+1} \right) \right\}$,
and so
\[ \left| \frac{d}{dt} \int u_n^2(t,x) g \left( \frac{x-A(\e)}{C(\e)} \right)
\right| \le \frac{\e}{T(\e)-T_0}. \] 
By integration on $[T_0,T(\e)]$~:
\[ \int_{x \ge 2C(\e) + A(\e)}  u_n^2(T_0,x) \le \int u_n^2(T_0,x) g \left
( \frac{x-A(\e)}{C(\e)} \right) \le 5\e. \] Now considering
$\frac{d}{dt} \int u_n^2(t,x) g \left ( \frac{-A(\e)-x}{C(\e)}
\right)$, we get in a similar way
\[ \int_{x \le -2C(\e) - A(\e)} u_n^2(T_0,x) \le 5\e. \]
Therefore, setting $A_\e = 2C(\e/10) + A(\e/10)$, we obtain :
\[ \forall n \in \m N, \qquad  \int_{|x| \ge A_\e} u_n^2(T_0,x)  \le \e.\]
\end{proof}

By \eqref{truc}, the sequence
$(u_n(T_0))$ is   bounded   in $H^1$, thus we can extract a subsequence (still denoted by $(u_n)$) which converges weakly to $\varphi_0 \in H^1(\m R)$.  The previous compactness result ensures that the convergence is strong in $L^2(\m R)$. Indeed, let
$\e >0$ and let $A$ be such that $\int_{|x| \ge A} \varphi_0^2(x) dx \le \e$
and
\[ \forall n \in \m N, \qquad  \int_{|x| \ge A} u_n^2(T_0,x)  \le \e. \]
By the compact embedding $H^1([-A,A]) \to L^2([-A,A])$,
$\int_{|x| \le A} | u_n(T_0,x)  - \varphi_0(x) |^2 dx \to 0$ as $n\to +\infty$.
We thus
derive that
\[ \limsup_{n \in \m N} \| u_n(T_0) - \varphi_0 \|_{L^2(\m R)}^2 \le 4 \e. \]
Since this is true for all $\e >0$, $u_n(T_0) \to \varphi_0$ in $L^2(\m
R)$  as $n\to +\infty$. By interpolation, $u_n(T_0)$ converges strongly to $\varphi_0$ in
$H^s$ for all $s \in [0,1)$.

\medskip

\noindent\emph{Step 2. Construction of the multi-soliton $u^*$.}

Denote $u^*(t)$ the solution to
\[ \left\{ \begin{array}{l}
u_t^* + (u_{xx}^* + (u^*)^{p})_x = 0, \\
u^*(T_0) = \varphi_0. \end{array} \right. \] 
Due to \cite{KPV93}, the Cauchy problem for (\ref{gkdv}) is locally well-posed in $H^{s}$ for $s \ge 1/2$~: we will work in $H^{1/2}$ (which is not a critical space) and $H^1$. Let $u^* \in \q C([T_0,T^*), H^{1})$ be the  maximal solution to (gKdV). Recall the  blow up alternative: either $T^*=+\infty$ or $T^* <\infty$ and then $\| u^*(t) \|_{H^1} \to \infty$ as $t \uparrow T^*$. 

Since the flow is continuous in $H^{1/2}$, for any $t \in [T_0,T^*)$, $u_n(t)$ is defined for $n$ large enough and $u_n(t) \to u^*(t)$ in $H^{1/2}$ as $n\to +\infty$. By the uniform $H^1$ bound, we also obtain  $u_n(t) \tendf u^*(t)$ in $H^1$-weak. Hence, using Proposition \ref{vndecay}, 
\[ \forall t \in [T_0,T^*), \qquad \| u^*(t) - R(t) \|_{H^1} \le \liminf_{n \to \infty} \| u_n(t) - R(t) \|_{H^1} \le C e^{-\sigma_0^{3/2} t}. \]
In particular, we deduce that
\[ \forall t \in [T_0,T^*), \quad \| u^*(t) \|_{H^1} \le  C e^{-\sigma_0^{3/2} t} + \| R(t) \|_{H^1} \le C + \sum_{j=1}^N \| Q_{c_j} \|_{H^1}. \]
Due to the blow-up alternative, it follows that $T^* =\infty$.
Hence $u^* \in \q C([T_0,\infty), H^1)$ and moreover $\| u^*(t) - R(t) \|_{H^1} \leq C e^{-\sigma_0^{3/2} t}$ for all  $t \geq T_0$.
\end{proof}

\subsection{Proof of Proposition \ref{vndecay}}\label{sec:2.3}

The proof proceeds in several steps. For the sake of simplicity, we will drop the index $n$  for the rest of this section (except for $S_n$). As  Proposition \ref{vndecay} is proved for given $n$, this should not be a source of confusion. Hence we will write $u$ for $u_n$, $\gh b_j^\pm$ for $\gh b_{j,n}^\pm$ etc.
We possibly drop the first terms of the sequence $S_n$, so that for all $n$, $S_n$ is large enough for our purposes.

\medskip

\noindent \emph{Step 1.} Choice of a set of initial data.

\begin{lem}[Modulation for time independent function] \label{modulation}
	Let $0<c_1<\ldots<c_N$.
There exist $C, \e>0$ such that the following holds. Given $(\alpha_i)_{i=1, \ldots N}$ such $\min \{ |\alpha_i -\alpha_j| \; \ {i \ne j}\} \ge 1/\e$, if $u(x) \in L^2$ is such that
\[ \left\| u - \sum_{j=1}^N Q_{c_j}(x-\alpha_j) \right\|_{L^2} \le \e, \]
then there exist modulation parameters $\g y = (y_j)_{j=1,\ldots,N}$ such that setting
\[v= u - \sum_{j=1}^N Q_{c_j}(x - \alpha_j - y_j), \]
the following holds
\begin{gather}
\left\| v \right\|_{L^2} + \sum_{j=1}^N | y_j | \le C \left\| u - \sum_{j=1}^N Q_{c_j}(x- \alpha_j) \right\|_{L^2},  \label{mod1}\\
\text{and } \ \forall j =1, \ldots, N, \quad \int v(x) (Q_{c_j})_x(x - \alpha_j- y_j) dx =0.\label{mod2}
\end{gather}
Furthermore, $u \mapsto (v,\g y)$ is a smooth diffeomorphism.
\end{lem}

\emph{Notation.} For  $\gh b$ small, from \eqref{gkdvSn} and continuity in $H^1$, $u(t)$ is defined and modulable (in the sense of the previous lemma) for $t$ close to $S_n$. As long as $u(t)$ is modulable around $R(t)$, we denote by $\g y(t) = (y_j(t))_{j=1,\ldots,N}$ the parameters of modulation, 
\begin{align*}
& \tilde R_j(t) = R_j(t, x-y_j(t)), \quad \tilde R(t) = \sum_{j=1}^N \tilde R_j(t), \quad \tilde Z^\pm_j(t,x) = Z^\pm_j(t,x-y_j(t)),\\
& v(t) = u(t) - \tilde R(t)\quad \text{so that}\quad
\forall j =1, \ldots, N, \quad \int v(t) (\tilde R_j)_x(t) =0,
\\ & \g a^{\pm}(t) = (a^{\pm}_j(t))_{j=1,\ldots,N}, \quad\text{where}\quad
a^{\pm}_j = \int v(t,x) \tilde Z^\pm_j(t,x) dx.
\end{align*}
% We also set $ L_j = - \partial_{xx} - p \tilde R_j^{p-1}(t) + c_j$.
We consider $\m R^N$ equipped with the $\ell^2$ norm. We denote by $B_{\q B}(P,r)$ the \emph{closed} ball of the Banach space $\q B$, centered at $P$ and of radius $r\geq 0$. If $P=0$, we simply write $B_{\q B}(r)$. Finally, $\m S_{\m R^N}(r)$ denotes the sphere of radius $r$ in $\m R^N$.

\bigskip

In view of Lemma \ref{ycoer}, we have to control the functions $\g a^{\pm}(t)$ on some time interval $[T_0,S_n]$. 
Since $Z^{+}$ and $Z^{-}$ are not orthogonal and because of the interactions between the various solitons, the values of $\g a^{\pm}(S_n)$ are not directed related to $\gh b$. The next lemma allows us to establish a one-to-one mapping between the choice of $\gh b$ in \eqref{gkdvSn} and the suitable constraints $\g a^{+}(S_n)= \gh a^+$, $\g a^{-}(S_n)= 0$, for any choice of $\gh a^+$.

\begin{lem}[Modulated final data] \label{finaldata}
	There exists $C>0$ (independent of $n$) such that
for all $\gh a^+ \in B_{\m R^N}(e^{-(3/2) \sigma_0^{3/2} S_n})$ there exists a unique $\gh b$ with $\|\gh b\|\leq C \|\gh a^+\|$ and such that the modulation $(v(S_n),\g y(S_n))$ of $u(S_n)$ satisfies
\begin{align*}
\g a^+(S_n) = \gh a^+ \quad \text{and} \quad \g a^-(S_n) = 0.
\end{align*}
\end{lem}

\begin{proof}
See Appendix.
\end{proof}

Let $T_0$ to be determined later in the proof, independent of $n$. Let $\gh a^+$ to be chosen, $\gh b$ be given by Lemma \ref{finaldata} and let $u$ be the corresponding solution of \eqref{gkdvSn}.
We now define the maximal time interval $[T(\gh a^+),S_n]$ on which   suitable exponential estimates hold. 

\begin{defi}\label{def}
Let $T(\gh a^+)$ be the infimum of  $T \ge T_0$ such that  the following properties hold for all $t \in [T,S_n]$~:
\begin{itemize}
\item[$\bullet$] Closeness to $R(t)$:
$$\| u(t) - R(t) \|_{H^1} \le \e.$$
In particular, this ensures that $u(t)$ is modulable around $R(t)$ in the sense of Lemma \ref{modulation}.
\item[$\bullet$] Estimates on the modulation parameters:
\begin{align*}
& e^{\sigma_0^{3/2} t} v(t)   \in B_{H^1}(1),\quad 
e^{\sigma_0^{3/2} t} \g y(t)   \in B_{\m R^N}(1), \\
& e^{(3/2) \sigma_0^{3/2} t} \g a^-(t)  \in B_{\m R^{N}}(1), \quad 
e^{(3/2) \sigma_0^{3/2} t} \g a^+(t)  \in B_{\m R^{N}}(1).
\end{align*}
\end{itemize}
\end{defi}

Observe that Proposition \ref{vndecay} is proved if for all $n$, we can find $\gh a^+$ such that $T(\gh a^+) = T_0$. The rest of the proof is devoted to prove the existence of such a value of $\gh a^+$.

\medskip

We claim the following  preliminary results on the modulation parameters of $u(t)$.

\begin{claim}
\begin{equation}\label{gkdvSnvn}  
v_t + \left( v_{xx} + (v + \tilde R)^{p} - \sum_{j=1}^N \tilde R_j^p \right)_x - \sum_{j=1}^N \frac {dy_j} {dt} {\/\tilde R_j}_x=0, 
\end{equation}
\begin{equation}\label{rien}
\forall t \in [T(\gh a^+), S_n], \qquad \left\| \frac{d\g y}{dt} (t) \right\| \le C \|v(t)\|_{L^2} + C e^{-2 \sigma_0^{3/2} t}.
\end{equation}
\begin{equation} \label{abound}
\forall t \in [T(\gh a^+), S_n], \ \forall j,\qquad \left| \frac{d a^\pm_j}{dt}(t) \pm e_0 c_j^{3/2} a^\pm_j(t) \right| \le C \| v(t) \|_{H^1}^2 + C e^{-3 \sigma_0^{3/2} t}.
\end{equation}
\end{claim}

\begin{proof}
The equation of 
	$v(t)$ is obtained by elementary computations from the equation of $u(t)$.
	Taking the scalar product of this equation with ${\/\tilde R_j}_x$, we see that $y_j(t)$ satisfy
	\begin{equation*}
	\frac{dy_j}{dt} \| {Q_{c_j}}_x  \|_{L^2}^2 = \int \left( v_{xx} + (v + \tilde R)^{p} - \sum_{j=1}^N \tilde R_j^p \right)_x {\/\tilde R_j}_x - \frac{dy_j}{dt}  \int v {\/\tilde R_j}_{xx}.
	\end{equation*} 
	From ($t \ge T_0$ large enough)
	$\| v(t) \|_{H^1} \le e^{- \sigma_0^{3/2} t} \le \frac{\|{Q_{c_j}}_x \|_{L^2}^2}{2 \| {Q_{c_j}}_{xx} \|_{L^2}},$
	 using integration by parts to have all the derivatives on ${\/\tilde R_j}_x$ and using Cauchy-Schwarz inequality, we get \eqref{rien}.
	\medskip
	
	Now, we prove \eqref{abound}.
	First, note that $\int {\/\tilde R_j}_x \tilde Z^\pm_j =0$ follows from  
\begin{equation} \label{Q_xZvanish}
\int Q_x Z^\pm =\pm e_0^{-1} \int Q_x L(Z^\pm_x) = \pm e_0^{-1} \int L(Q_x) Z^\pm_x =0.
\end{equation}
Using the equation of $v(t)$ and next the equations of $Z^\pm$,
\begin{align*}
& \frac{da_j^\pm}{dt} (t)  = \int v_t \tilde Z_j^\pm + \int v {\/\tilde Z_j^\pm}_t \\
& = - \int (v_{xx} + (v+ \tilde R)^{p} -\sum_k \tilde R_k^p)_x \tilde Z_j^\pm + \sum_k \frac{dy_k}{dt} \int {\/\tilde R_k}_x \tilde Z^\pm_j 
- ( c_j + \frac{dy_j}{dt} ) \int v {\/\tilde Z_j^\pm}_x\\
& = - \int (v_{xx} + p \tilde R_j^{p-1} v)_x \tilde Z_j^\pm - c_j \int v {\/\tilde Z_j^\pm}_x \\
& \qquad \qquad + \int ((v+\tilde R)^{p} - \sum_k \tilde R_k^p - p \tilde R_j^{p-1}v)_x \tilde Z_j^\pm +  \sum_{k\ne j} \frac{dy_k}{dt} \int {\/\tilde R_k}_x \tilde Z^\pm_j - \frac{dy_j}{dt} \int v {\/\tilde Z_j^\pm}_x\\ 
& = - \int v L_j ({\/\tilde Z_j^\pm}_x) - \int ((v+ \tilde R)^{p} - \sum_k \tilde R_k^p - p \tilde R_j^{p-1}v) {\/\tilde Z_j^\pm}_x + \sum_{k \ne j} \frac{dy_k}{dt} \int {\/\tilde R_k}_x \tilde Z^\pm_j -  \frac{dy_j}{dt} \int v {\/\tilde Z_j^\pm}_x\\
& = \mp e_0 c_j^{3/2} a_j^\pm(t) - \int ((v+ \tilde R)^{p} - \sum_k \tilde R_k^p - p \tilde R_j^{p-1}v) {\/\tilde Z_j^\pm}_x + \sum_{k \ne j} \frac{dy_k}{dt} \int {\/\tilde R_k}_x \tilde Z^\pm_j -  \frac{dy_j}{dt} \int v {\/\tilde Z_j^\pm}_x.
\end{align*}
Using \eqref{sigma0}, for $k \ne j$,
\begin{equation} 
|\tilde R_k(t,x)| (|\tilde Z_j^\pm(t,x)| + |{\/\tilde Z_j^\pm}_x(t,x)| )\le C e^{- 2\sqrt{\sigma_0} (|x- c_k t| + |x-c_j t|)} \le C e^{- 3 \sigma_0^{3/2} t} e^{-\sqrt{\sigma_0} |x-c_jt|}.
\end{equation}
Hence we have
\begin{gather}
\left| \int (|v+\tilde R|^{p-1}(v+\tilde R) - \sum_{k=1}^N \tilde R_k^p - p\tilde R_j^{p-1} v) {\/\tilde Z_j^\pm}_x \right| 
\le C \| v(t) \|_{H^1}^2 + C e^{-3\sigma_0^{3/2} t},  \label{nonlinbound1}\\
\left| \sum_{k \ne j} \frac{dy_k}{dt} \int {\/\tilde R_k}_x \tilde Z^\pm_j \right| \le C e^{- 3 \sigma_0^{3/2} t} \| v(t) \|_{H^1} 
\le C \| v(t) \|_{H^1}^2 + C e^{- 4 \sigma_0^{3/2} t}. \label{nonlinbound2}
\end{gather}
The term $\frac{dy_j}{dt} \int v {\/\tilde Z_j^\pm}_x$ is controlled using \eqref{rien}.
\end{proof}

\noindent\emph{Step 2.} Conditionnal stability of $ v$ and $\g y$ under the control of $\g a^\pm$.

We claim the following improvement of the estimates for $v(t)$ and $\g y$ on $[T(\gh a^+),S_n]$.

\begin{lem}[Control of $v$ and $\g y$]\label{vcontrol}
For   $T_0$ large enough (independent of $n$) and for all $\gh a^+ \in B_{\m R^N}(e^{-(3/2)\sigma_0^{3/2} S_n})$, the following holds
\begin{align} 
	\forall t\in [T(\gh a^+),S_n],\quad 
& \| u(t) - R(t) \|_{H^1}  \le C e^{-\sigma_0^{3/2} t} \le \e_0/2, \label{undecay} \\
& e^{\sigma_0^{3/2} t} \| v(t) \|_{H^1}   \le 1/2,\quad
e^{\sigma_0^{3/2} t} \| \g y(t) \|   \le 1/2.\label{vcont}
\end{align}
\end{lem}

The proof of Lemma \ref{vcontrol} is postponed to the end of this section. It is very similar to the proofs in the subcritical case (see \cite{Mar05} or \cite{MM06}).

\bigskip

\noindent\emph{Step 3.} Control of   $\g a^-(t)$.

\begin{lem}[Control of $\g a^-(t)$]\label{le:a-}
For   $T_0$ large enough (independent of $n$)  and for all $\gh a^+ \in B_{\m R^N}(e^{-(3/2)\sigma_0^{3/2} S_n})$, the following holds
$$\forall t\in [T(\gh a^+),S_n],\quad e^{(3/2) \sigma_0^{3/2} t} \| \g a^-(t) \| \le 1/2.$$
\end{lem}

\begin{proof}
It follows from \eqref{abound}, \eqref{vcont} and $a_j^-(S_n)=0$ that for all $t \in [T(\gh a^-),S_n]$,
\begin{align*}
|a_j^-(t)| & \le C e^{e_0 c_j^{3/2} t} \int_t^{S_n} e^{-e_0 c_j^{3/2} s} \left( e^{-2 \sigma_0^{3/2} s}
+ e^{- 3 \sigma_0^{3/2} s} \right) ds \le C e^{-2 \sigma_0^{3/2} t}.
\end{align*}
Hence, for $T_0$  large enough, 
$ \forall t \in [T(\gh a^-),S_n],$ $\|\g a^-(t)\| \le C e^{-2 \sigma_0^{3/2} t} \le \frac{1}{2} e^{-(3/2) \sigma_0^{3/2} t}$.
\end{proof}

\noindent\emph{Step 4.} Control of $\g a^+(t)$ by a topogical argument.

Finally we turn to the control of $\g a^+(t)$  which will provide us with a suitable value of $\gh a^+$. This is the new key argument of this paper.

\begin{lem}[Control of $\g a^+(t)$]\label{le:a+}
For  $0<\sigma_0 <\bar \sigma_0$ small enough, $T_0$ large enough,
there exists $\gh a^+ \in B_{\m R^N}( e^{-(3/2) \sigma_0^{3/2} S_n})$  such that $T(\gh a^+) = T_0$.
\end{lem}

\begin{proof}
We argue by contradiction. Assume that for all $\gh a^+ \in B_{\m R^N}( e^{-(3/2) \sigma_0^{3/2} S_n})$, one has $T(\gh a^+) > T_0$. From Lemmas \ref{vcontrol} and \ref{le:a-}
\begin{align*}
& u(T(\gh a^+)) - R(T(\gh a^+))  \in B_{H^1}(\e_0/2) , \quad
e^{\sigma_0^{3/2} T(\gh a^+)} v(T(\gh a^+))  \in B_{H^1}(1/2) , \\ 
& e^{\sigma_0^{3/2} T(\gh a^+)} \g y(T(\gh a^+))  \in B_{\m R^N}(1/2) , \quad
e^{(3/2) \sigma_0^{3/2} T(\gh a^+)} \g a^-(T(\gh a^+))  \in B_{\m R^N}(1/2).
\end{align*}
Hence by definition of $T(\gh a^+)$ and continuity of the flow, one must have 
\begin{equation}\label{contra}
	e^{(3/2) \sigma_0^{3/2} T(\gh a^+)} \g a^+(T(\gh a^+)) \in  \m S_{\m R^N}(1).
\end{equation}

Let  $T<T(\gh a^+)$ be close enough to $T(\gh a^+)$ so that the solution $u(t)$ and its modulation are well-defined on $[T,S_n]$.
For $t\in [T,S_n]$, let
\begin{equation}\label{eq:N} 
	\mathcal{N}(\gh a^+(t))= \mathcal{N}(t)=  \left\| e^{(3/2) \sigma_0^{3/2} t} \g a^+(t) \right\|^2.
\end{equation}
Then, by \eqref{abound} and \eqref{vcont}, we have
\begin{equation}\label{eq:Np}
	\left|	\frac d{dt} \mathcal{N}(t) + (  2 e_0 c_j^{3/2}-3 \sigma_0^{3/2} ) \mathcal{N}(t)	\right| \leq
	C  e^{- (3/2) \sigma_0^{3/2} t} (\|v(t)\|_{L^2}^2 + e^{-3 \sigma_0^{3/2} t } )  \leq 
	C e^{- (1/2) \sigma_0^{3/2} t} .
\end{equation}
In particular, in view of the definition of $\sigma_0$ (see \eqref{sigma0}), for all $j$, $2 e_0 c_j^{3/2}-3 \sigma_0^{3/2} \geq e_0 c_1^{3/2}\geq 4 e_0 \sigma_0^{3/2},$
applying the previous estimate at $t=T(\gh a^+)$, and using $\mathcal{N}(T(\gh a^+))=1$, we get
\begin{equation}\label{eq:Nn}
	\forall \gh a^+\in B_{\m R^N}( e^{-(3/2) \sigma_0^{3/2} S_n}),\quad
	\frac d{dt} \mathcal{N}(T(\gh a^+))  \leq - 4 e_0 \sigma_0^{3/2}.
\end{equation}

From \eqref{eq:Nn}, a standard argument says that the map $\gh a^+ \mapsto T(\gh a^+)$ is continuous. 
Indeed, by \eqref{eq:Nn}, for all $\epsilon>0$, there exists $\delta>0$ such that $\mathcal{N}(T(\gh a^+) - \e) > 1+\delta$ and $\mathcal{N}(T(\gh a^+) + \e) < 1-\delta$.
By continuity of the flow of the (gKdV) equation, it follows that there exist $\eta >0$ such that for all $\| \tilde {\gh a}^+ -\gh a^+\|\leq  \eta$, 
the corresponding  $\tilde {\g a}^+(t)$ satisfies $|\mathcal{N}(\tilde {\g a}^+(t)) - \mathcal{N}(\g a^+(t))|\leq  \delta/2$ for all $t \in [T(\gh a^+)-\epsilon, S_n]$. In particular, $T(\gh a^+) - \epsilon \le T(\tilde {\gh a}^+) \le T(\gh a^+) + \epsilon$.

Now, we consider the continuous map
\begin{equation*}\begin{split}
	\mathcal{M} \ : \quad  B_{\m R^N}( e^{-(3/2) \sigma_0^{3/2} S_n}) & \to \m S_{\m R^N}( e^{-(3/2) \sigma_0^{3/2} S_n}),\\
		\gh a^+ & \mapsto  e^{-(3/2) \sigma_0^{3/2} (S_n-T(\gh a^+))} \g a^+(T(\gh a^+)) ).
\end{split}\end{equation*}
Let $\gh a^+ \in \m S_{\m R^N}(e^{-(3/2) \sigma_0^{3/2} S_n})$. From \eqref{eq:Nn}, it follows that $T(\gh a^+) = S_n$ and
so $\mathcal{M}(\gh a^+)=\gh a^+$, which means that $\mathcal{M}$ restricted to $\m S_{\m R^N}(e^{-(3/2) \sigma_0^{3/2} S_n})$ is the identity.
But the existence of such a map $\mathcal{M}$ contradicts Brouwer's fixed point theorem.

In conclusion, there exists $\gh a^+ \in B_{\m R^N}( e^{-(3/2) \sigma_0^{3/2} S_n})$ such that $T(\gh a^+) = T_0$.
\end{proof}

The end of this section is devoted to the proof of Lemma \ref{vcontrol}.

\begin{proof}[Proof of Lemma \ref{vcontrol}]
Define 
\begin{equation} \label{psi0}\begin{split}
& \psi (x) = \frac 2 \pi \arctan(\exp(- \sqrt{\sigma_0} x)), \quad \psi_j(t,x) = \psi\left(\frac 1{\sqrt{t}} (x - m_j(t))\right), 	\quad 
	\psi_N (t) =1,\\
& \text{where  for $j=
1,\ldots,N-1$}, \quad
m_j(t) = \frac 12 \left((c_j + c_{j+1})  t +  y_j + y_{j+1}  \right);
\\
& \phi_1 = \psi_1, \quad \phi_j = \psi_{j} - \psi_{j-1}, \quad \text{for $j=
1,\ldots,N$;}\\
& M_j(t) = \int u^2(t) \phi_j(t), \quad E_j(t) = \int \left( \frac{1}{2} u_x^2 - \frac{1}{p+1} u^{p+1} \right)(t) \phi_j(t). 
\end{split}\end{equation}
We begin with some technical claims.

\begin{claim}  
\begin{equation}\label{flocal}
\left| \frac{d}{dt} M_j(t)  \right| \le \frac{C}{\sqrt{t}} \|  v(t) \|_{H^1}^2 + C e^{- 3 \sigma_0^{3/2} t},
\end{equation}
\begin{equation} \label{localcontrol}
\left| \frac{d}{dt}  \sum_{j=1}^N  \left(E_j(t) + \frac{c_j}{2} M_j(t) \right)  \right|\le \frac{C}{\sqrt{t}} \|  v(t) \|_{H^1}^2 + C e^{-3 \sigma_0^{3/2} t}.
\end{equation}
\end{claim}

\begin{proof} By direct computations,
\begin{align*}
\frac{d}{dt} \int u^2 \phi_j & = -3 \int u_x^2 {\phi_j}_x + \int u^2 \left({\phi_j}_{xxx} + {\phi_j}_t \right) + \frac{2p}{p+1} \int u^{p+1} {\phi_j}_x.
\end{align*}
By the decay properties of  $\phi_j(t)$ and  $\tilde R_j(t)$, for all $k$,
\begin{equation}\label{mach}
	| \tilde R_k| ( |{\phi_j}_x | + |{\phi_j}_{xxx} | + |{\phi_j}_t| ) \le Ce^{-3 \sigma_0^{3/2} t} e^{- \sigma_0 |x -c_k t -x_k|}\; \quad |{\phi_j}_x | + |{\phi_j}_{xxx} | + |{\phi_j}_t| \le \frac{C}{\sqrt{t}}.
\end{equation}
Thus, expanding $u(t) = \tilde R(t) +  v(t)$,  the first two integrals are estimated as desired. For the last term it suffices to observe that $\|  u(t) \|_{L^\infty} \le C ( \| v(t) \|_{H^1} + \|  \tilde R(t) \|_{H^1})\leq C $. This proves \eqref{flocal}.

Estimate  \eqref{localcontrol} is a consequence of \eqref{flocal},  the conservation of energy and  $\sum_{j=1}^N \phi_j =1$.
\end{proof}

\begin{claim} 
\begin{equation}\label{mainterms3}
  \left| \left( E_j(t) +
\frac{c_j} 2 M_j(t) \right) - \left(E(Q_{c_j}) + \frac{c_j} 2 \int
Q_{c_j}^2 \right) - \frac 1 2 H_j(t) \right|   \le  C  e^{- 3 \sigma_0^{3/2} t}
+  C e^{-\sigma_0^{3/2} t}  \|v(t)\|_{L^2}^2,  
\end{equation}
where $\displaystyle H_j(t) = \int (v_x^2(t) - p \tilde
R_j^{p-1} (t) v^2(t) + c_j v^2(t) ) \phi_j(t)$.
\end{claim}

\begin{proof}
	First, we claim
	\begin{align}
	 &\left| M_j(t) - \left( \int Q_{c_j}^2 + 2 \int
	v(t) \tilde R_j(t) + \int  v^2(t) \phi_j(t) \right) \right|
	\le C  e^{- 3 \sigma_0^{3/2} t},    \label{mainterms1} \\
	& \left| E_j(t) - E(Q_{c_j}) - \left( \frac 1 2 \int
	(v_x^2(t) - p \tilde R_j^{p-1} (t)
	v^2(t) ) \phi_j(t) - c_j \int v(t)
	\tilde R_j(t) \right) \right|  \nonumber \\
	& \quad  \le 
	C  e^{- 3 \sigma_0^{3/2} t} + C e^{-\sigma_0^{3/2} t} \|v(t)\|^2_{L^2},  \label{mainterms2} 
	\end{align}
	Indeed,
expanding $u(t)=v(t)+\sum_k \tilde R_k(t)$ in $M_j(t)$  and $E_j(t)$, we get
\[ M_j(t) = \int u_n^2 \phi_j(t) = \int \left( v^2 + 2 v
\tilde R + \sum_{k=1}^N \tilde R_k^2 \right) \phi_j(t), \]
\begin{align*}
E_j(t) & = \int  \left( \frac 1 2 ( v_x^2 + 2
v_x \tilde R_x + \tilde R_x^2 ) - \frac 1 {p+1} (v +
\tilde
R)^{p+1} \right) \phi_j(t) \\
& = \int \left ( \frac 1 2 (v_x^2 - p \tilde R^{p-1}
v^2)\right) \phi_j +
\int \left( \frac  1 2 \tilde R_x^2 - \frac 1 {p+1} \tilde R^{p+1} \right)
\phi_j(t) \\
& \quad  - \int v (\tilde R_{xx} + \tilde R^p) \phi_j
- \int \tilde R_x v {\phi_j}_x \\
& \quad  + \frac 1{p+1} \int \left(  (-(v + \tilde R)^{p+1} +
\tilde R^{p+1}) + (p+1)v \tilde R^p +  (p+1)p \tilde
R^{p-1} v^2 \right)  \phi_j(t).
\end{align*}

By the decay properties of  $\phi_j(t)$ and  $\tilde R_j(t)$ we have ($k \ne j$)
\[ \left| \int \tilde R_j^2 \phi_j(t) - \int Q_{c_j}^2 \right|
+ \int \tilde R_k^2
\phi_j(t) +
\left| \int \left( \frac  1 2 \tilde R_x^2 - \frac 1 {p+1} \tilde R^{p+1} \right)
\phi_j(t) - E(Q_{c_j})\right|\leq C e^{-3\sigma_0^{3/2} t}.
\]
By $Q_{xx} + Q^p = Q$, we have
\[  \int v(t) (\tilde R_{xx} + \tilde R^p) \phi_j = c_j
\int v(t) \tilde R_j(t) + O(e^{- 3 \sigma_0^{3/2}
t}). \]
Using also \eqref{mach} and for $k\geq 3$
\[ \int |v(t)|^k \phi_j(t) \le \| v(t) \|_{L^\infty}^{k-2} \int
v(t)^2 \phi_j(t)\leq  C e^{- 3 \sigma_0^{3/2} t} \|v\|_{L^2}^2,\]
we obtain  \eqref{mainterms1} and \eqref{mainterms2}.

\medskip

\noindent Estimate \eqref{mainterms3} is obtained by summing   \eqref{mainterms1} and
\eqref{mainterms2}. Note that in particular that the scalar 
products $\int v(t) \tilde R_j(t)$ cancel.
\end{proof}

\begin{claim}
\begin{equation}\label{variation}
	\exists K>0,\ \forall v\in H^1,\quad
	\|v(t)\|_{H^1}^2 \leq K \sum_{j} H_j(t) + K^2 \sum_{j} \left ( \left( \int v(t) \tilde Z^+_j(t) \right)^2 + \left( \int v(t) \tilde Z^-_j(t)
	\right)^2 \right).
\end{equation}
\end{claim}

\begin{proof} Estimate \eqref{variation} is a standard consequence of Lemma \ref{ycoer} and $\int v {\/ \tilde R_j}_x= 0$.
See e.g. \cite[Lemma 4]{MMT02}.
\end{proof}

Now, we finish the proof of  Lemma \ref{vcontrol}. Let $t\in [T(\gh a^+),S_n]$.
Integrating \eqref{localcontrol} on $[t,S_n]$,
\[ \left| \sum_{j=1}^N \left\{ \left(E_j(S_n)+\frac {c_j}2 M_j(S_n)\right)   - \left(E_j(t)+\frac {c_j}2 M_j(t)\right) \right\} \right| \le C e^{-3 \sigma_0^{3/2} t} + C \int_t^{S_n} \| v(s) \|_{H^1}^2 \frac{ds}{\sqrt{s}}. \]
From \eqref{mainterms3}, we get :
\[ \left| \sum_{j=1}^N (H_j(S_n) - H_j(t)) \right| \le  C e^{-3 \sigma_0^{3/2} t} + C e^{-\sigma_0^{3/2}t}( \| v(t) \|_{L^2}^2 + \| v(S_n) \|_{L^2}^2) + C \int_t^{S_n} \| v(s) \|_{H^1}^2 \frac{ds}{\sqrt{s}}. \]
Note that from Lemmas \ref{modulation} and \ref{finaldata}, and from the definition of $T(\gh a^+)$,
\[ |H_j(S_n)| \le C \|v(S_n)\|_{H^1}^2 \leq C \|\gh b\|^2  \le C e^{-3 \sigma_0^{3/2} t} \quad \text{and}\quad  \| v(t) \|_{L^2}^2 \le C e^{-2\sigma_0^{3/2} t}. \]
 
By   \eqref{variation} and the above estimates
\begin{align}
& \| v(t) \|_{H^1}^2   \le K \sum_{j=1}^N H_j(t) + K^2 \sum_{j,\pm} a_j^\pm(t)^2 \nonumber \\
& \le C e^{-3 \sigma_0^{3/2} t} + C \sum_{j,\pm} a_j^\pm(t)^2 + C \int_t^{S_n} \| v(s) \|_{H^1}^2 \frac{ds}{\sqrt{s}}  \le C_0 e^{-3 \sigma_0^{3/2} t} + \frac{C_0}{\sqrt{t}} e^{-2 \sigma_0^{3/2} t}. \label{vdecay}
\end{align}
Hence, for $T_0$ large enough so that  $C_0e^{- \sigma_0^{3/2} T_0} \le 1/8$ and  $C_0/{\sqrt{T_0}} \le 1/8$  we get
\[ e^{\sigma_0^{3/2} t} \| v(t) \|_{H^1} \le 1/2. \]
By \eqref{rien} and \eqref{vdecay},
\begin{align}
\|\g y_t(t)\| & \le C e^{-  2 \sigma_0^{3/2} t} + C \|v(t)\|_{L^2},   \nonumber \\
\|\g y(t)\| & \le |\g y (S_n)| + C \int_{t}^{S_n} \left( e^{- (3/2) \sigma_0^{3/2} s} + \frac{e^{-\sigma_0^{3/2} s}}{\sqrt{s}} \right) ds \le C e^{- (3/2) \sigma_0^{3/2} t} + \frac{C}{\sqrt{t}} e^{-\sigma_0^{3/2} t}, \label{ydecay}
\end{align}
and we deduce $e^{\sigma_0^{3/2} t} \| \g y(t) \| \le 1/2$ by possibly taking a larger $T_0$.
Finally, we have~:
\begin{align}
\| u(t) - R(t) \|_{H^1} & \le \| R(t) - \tilde R(t) \|_{H^1} + \| v(t) \|_{H^1} \le C \|\g y(t)\| + \| v(t) \|_{H^1} \nonumber \\
& \le C e^{-\sigma_0^{3/2} t} \le \e_0/2, \label{udecay}
\end{align}
by possibly taking a larger $T_0$. This concludes the proof of Lemma \ref{vcontrol}.
\end{proof}

\section{Generalizations}

\subsection{The gKdV equations with general nonlinearities}\label{sec:gkdv}

We now present  extensions of Theorem \ref{nsoliton} to a  more general form of the KdV equation, i.e.
\begin{equation}
\label{fkdv} \tag{gKdV}
u_t + (u_{xx} + f(u) )_x =0, \quad (t,x) \in \m R \times \m R.
\end{equation}

In order to have both well-posedness in $H^1$ from \cite{KPV93} and the existence of eigenvalues for the linearized operator in the instable case from \cite{PW92}, we assume
\begin{equation}\label{assump}
	\text{$f$ is $C^3$, convex for $u>0$, and $f(0)=f'(0)=0$,}
\end{equation}
but these assumptions can probably be relaxed.
Concerning the solitons, 
we consider velocities $c_j>0$ such that 
\begin{equation}\label{nondeg}
	\text{a solution $Q_c$ of $(Q_c)_{xx} + f(Q_c) = cQ_c$ exists for all $c$ close to $c_j$ and} \quad {\frac d{dc} \int Q_c^2}_{|c=c_j}\ne 0.
\end{equation}
Then, combining the proof of Theorem \ref{nsoliton} and \cite{Mar05}, we claim the following extension of Theorem~\ref{nsoliton}.

\begin{thm}\label{ggkdv}
Let $0<c_1<\ldots < c_N$ and $x_1, \ldots, x_N \in \m R$ be such that for all $j$, \eqref{nondeg} holds.
There exist $T_0 \in \m R$, $C,\sigma_0>0,$ and a solution $u \in \q C([T_0,\infty), H^1)$ to \eqref{gkdv} such that
\[ \forall t \in [T_0,\infty), \qquad \left\| u(t) - \sum_{j=1}^N R_{c_j,x_j}(t) \right\|_{H^1} \le C e^{-\sigma_0^{3/2} t}. \]
\end{thm}

\begin{nb}
The critical case ${\frac d{dc} \int Q_c^2}_{|c=c_j}= 0$ is treated in \cite{Mar05} for the pure power case. We leave open the special case where for
a general $f(u)$, ${\frac d{dc} \int Q_c^2}_{|c=c_j}= 0$ for some $c_j$, but it probably can be treated by similar techniques.
 
From the techniques developped in \cite{Miz04}, \cite{Dik05} and \cite{DiM04} concerning the (BBM) equation 
\begin{equation}
\label{BBM} \tag{BBM}
	(u-u_{xx})_t + (u + u^p )_x =0, \quad (t,x) \in \m R \times \m R, 
\end{equation}
and from the construction of suitable eigenfunctions of the linearized equation by Pego and Weinstein \cite{PW92}  (see page 74),
one can also extend the results obtained in this paper to the (BBM) equation for any $p>1$.
\end{nb}

\subsection{The non linear Schr\"odinger equations}\label{sec:nls}

In this section, we sketch the proof of Theorem \ref{nsolitonnls}.
It is an extension of  the proof of Theorem~\ref{nsoliton} in the present paper and of the main result in \cite{MM06}.

\subsubsection{Preliminaries}
Let $v=v_1+iv_2$, we define the operator $\mathcal{L}$ by 
\begin{equation}
{\cal L} v  = - L_- v_2 + i L_+ v_1,
\end{equation}
where the self-adjoint operators $L_+$ and $L_-$ are defined by
\begin{equation}
\label{defL+L-}
L_+v_1:=-\Delta v_1+v_1-p Q^{p-1} v_1,\quad L_- v_2:=-\Delta v_2+v_2-Q^{p-1} v_2, 
\end{equation}
From \cite{Wei85}, \cite{Gri88} and  \cite{Schla06}, there exist $e_0>0$, ${\  Y}^\pm \in {\cal S}(\m R)$ (${\bar Y}^+ =Y^-$), normalized so that
$\|Y^\pm\|_{L^2}=1$ and such that
\begin{equation}\label{Ynls}
	{\cal L} Y^\pm  = \pm e_0 Y^\pm;
\end{equation}
moreover, for some $K>0$,
 for any $v=v_1+iv_2 \in H^1$ ($(f,g)=\Re \int f\bar g$)
\begin{equation}\label{orthoef}\begin{split}
\|v\|_{H^1}^2 & \leq K (L_+ v_1,v_1 ) +  K (L_- v_2, v_2) \\
& + K^2 \left( \int (\nabla Q) v_1\right)^2 + K^2 \left(\int Q v_2\right)^2+
  K^2 \left(\Im \int Y^+ \bar v \right)^2 + K^2 \left(\Im \int   Y^- \bar v \right)^2.
\end{split}\end{equation}
See  \cite{DM07,DRar} for the proof of \eqref{orthoef}.

\subsubsection{Proof of Theorem \ref{nsolitonnls} assuming uniform estimates}
We denote
\begin{equation}\label{Rnls}\begin{split}
&  R(t,x) = \sum_{j=1}^N R_j(t,x) \quad \text{where} \quad	R_j(t,x)  = R_{c_j,\gamma_j, v_j,x_j}(t,x),
\\
& Y^\pm_j(t,x) = c_j^{\frac{1}{p-1}} Y^\pm(\sqrt{c_j}(x - v_j t - x_j)) e^{i (\frac 12 v_j\cdot x - \frac 14 \|v_j\|^2 t + c_j t + \gamma_j)}.  
\end{split}\end{equation}
Let $S_n \to \infty$ be an increasing sequence of time. We claim the existence of final data giving suitable uniform estimates.
\begin{prop} \label{vndecaynls}
There exist $n_0\geq 0$, $\sigma_0>0, T_0>0, C >0$ (independent of $n$) such that the following holds. 
For each $n\geq n_0$, there exists $\gh b = (\gh b_{j,n}^\pm)_{j,\pm} \in \m R^{2N}$ with 
$\|\gh b\| \le  e^{-\sigma_0^{3/2} S_n},$
and such that the solution $u_n$ to 
\begin{equation} 
\label{nlsSn}
\left\{ \begin{array}{>{\displaystyle}l}
 i {u_n}_t + \Delta {u_n}  + |u_n|^{p-1}u_n =0, \vphantom{\sum_{j \in \{ 1,\ldots, N\}, \pm}} \\
u_n(S_n) = R(S_n) +  i \sum_{j \in \{ 1,\ldots, N\}, \pm} \gh b_{j,n}^\pm Y_j^\pm(S_n)
\end{array} \right.
\end{equation}
is defined on the interval $[T_0, S_n]$, and satisfies
\[ \forall t \in [T_0, S_n], \quad \left\| u_n(t) - R(t) \right\|_{H^1} \le C e^{-\sigma_0^{3/2} t}. \]
\end{prop}

The proof of Theorem \ref{nsolitonnls} assuming Proposition \ref{vndecaynls} is completely similar to Section 2.2 in the present paper and to Section 2 in \cite{MM06}, thus it is omitted (note that for this part, as in \cite{MM06}, we use the local $H^s$ Cauchy theory due to Cazenave and Weissler \cite{CW90}).

\subsubsection{Proof of the uniform estimates}
We are reduced to prove Proposition \ref{vndecaynls}.
We only sketch the proof since it is  very similar to Section 2.3 of the present paper combined with Section 3 in \cite{MM06}.

The first step of the proof is to reduce (without loss of generality) to the special case where
$$
	v_{1,1} < v_{2,1} < \ldots < v_{N,1},
$$
where $v_{j,k}$ ($j\in \{1,\ldots,N\}$, $k\in \{1,\ldots,d\}$) represents the $k-th$ component of the velocity vector $v_j\in \m R^{d}$.
It is a simple observation, based on the invariance by rotation of the (NLS) equation, see Claim 1, page 855 of \cite{MM06}. 

Next, in the (NLS) case,  modulation theory for $u(t)$ close to $R(t)$ says that there exist parameters $\g y(t)=(y_1(t),\ldots,y_N(t))\in (\m R^d)^N$
and $\mu(t)=(\mu_1(t),\ldots,\mu_N(t))\in \m R^N$ such that
\begin{align*}
& \tilde R_j(t) = R_j(t, x-y_j(t)) e^{i \mu_j(t)}, \quad \tilde R(t) = \sum_{j=1}^N \tilde R_j(t), \quad \tilde Y^\pm_j(t,x) = Y^\pm_j(t,x-y_j(t))e^{i \mu_j(t)},\\
& v(t) = u(t) - \tilde R(t)\quad \text{satisfies}\quad
\forall j =1, \ldots, N, \quad \Re \int v(t) (\nabla \tilde R_j)(t) = \Im \int v(t) \tilde R_j(t)=0,
\end{align*}
Note that the  phase parameter $\mu_j(t)$ is used to control the direction $\Im \int v(t) \tilde R_j(t)$.

In view of \eqref{orthoef}, we are led to set
\begin{equation*}
 \g a^{\pm}(t) = (a^{\pm}_j(t))_{j=1,\ldots,N}, \quad\text{where}\quad
	a^{\pm}_j(t) = \Im \int  \tilde Y^\mp_j(t,x) \bar v(t,x) dx.
\end{equation*}

For  given $\gh a^+ \in \m R^{N}$, we define $\gh b \in \m R^{2N}$ as for the (gKdV) case in Lemma \ref{finaldata}.
We define $T(\gh a^+)$ as in Definition \ref{def}, with  the additional requirement $e^{\sigma_0^{3/2} t}\mu(t)\in B_{\m R^N}(1)$.
By  standard computations, the following holds on $[T(\gh a^+),S_n]$.

\begin{claim} For some $\sigma_0>0$,
	\begin{equation}\label{Ytnls}
		\left\| \frac{d\g y}{dt} (t) \right\| + \left\| \frac{d\mu}{dt} (t)\right\| \leq C \|v(t)\|_{L^2} + C e^{-2 \sigma_0^{3/2}t},
	\end{equation}
	\begin{equation}\label{Anls}
		\left| \frac{da_j^{\pm}}{dt} (t) \pm e_0 c_j^{3/2}a_j^{\pm}(t)\right| \leq C \|v(t)\|_{L^2}^2 + C e^{-3 \sigma_0^{3/2}t}.
	\end{equation}
\end{claim}
\begin{proof} The proof follows 
	 from the equation of $v$
		\begin{equation*}
		%	\begin{split} &
			i v_t + \Delta v + \sum_{j} \left(|\tilde R_j|^{p-1} v + (p-1) |\tilde R_j|^{p-2} \Re (\tilde R_j v) \right) + O(\|v\|_{H^1}^2)  - \sum_j \frac{dy_j}{dt} {\/\tilde R_j}_x  - i \sum_j \frac{d\mu_j}{dt} \tilde R_j =0,\\
	%	&	\text{where}\quad		{\cal E}(t)= - i |\tilde R + v|^{p-1} (\tilde R+v) + i \sum_j \tilde R_j^p + i p \sum_j \tilde R_j^{p-1} v_1 - \sum_j \tilde R_j^{p-1} v_2. 	
	%\end{split}
	\end{equation*}
	and direct computations using the definition of $Y^{\pm}$.
\end{proof}

Now we follow exactly the same strategy as in the proof of Theorem \ref{nsoliton}, by proving analogues of 
Lemmas \ref{vcontrol}, \ref{le:a-} and \ref{le:a+}. 

For the proof of the estimate on $v(t)$, we use a functional adapted to the (NLS) equations, as in \cite{MM06} and \cite{MMT06}:
$$
\mathcal{G}(t) =	\sum_{j} \left(\int \left(\frac 12 |\nabla u|^2 -\frac 1{p+1} |u|^{p+1}\right) \phi_j
	+ \left ( c_j + \frac {|v_j|^2}4 \right)\int |u|^2 \phi_j - v_j \cdot \Im \int  \bar u \nabla u  \phi_j  \right),
$$
where
\begin{align*}
& \psi_j(t,x)= \psi\left( \frac 1{\sqrt{t}} (x_1 - m_j(t))\right),\quad
m_j(t)= \frac 12 ( (v_{j,1}+v_{j+1,1}) t + y_{j,1} + y_{j+1,1});\\
& \phi_1= \psi_1,\quad  \phi_j =  \psi_j -  \psi_{j-1}.
\end{align*}
Note that $\mathcal{G}(t)$ controls the size of $v(t)$ in $H^1$ up to $\g a^{\pm}(t)$ as a consequence of \eqref{orthoef}.
As for (gKdV), the following claim allows us to prove the estimate on $\|v(t)\|_{H^1}$.
\begin{claim}  
\begin{equation*}
\left| \frac{d\mathcal{G}}{dt} (t)  \right| \le \frac{C}{\sqrt{t}} \|  v(t) \|_{H^1}^2 + C e^{- 3 \sigma_0^{3/2} t},
\end{equation*}
\end{claim}

The estimates of $\g a^{\pm}(t)$ are exactly the same as in Lemmas \ref{le:a-} and \ref{le:a+}, using \eqref{Anls}.

\appendix

\section{Appendix}

\begin{proof}[Proof of Lemma \ref{modulation}]
We use the following  notation $\g y=(y_j)_{j=1,\ldots,N}$ and 
\[ R_j(x) = Q_{c_j}(x-\alpha_j), \quad \tilde R_j(x) = R_j(x-y_j), \quad   \quad R(x) = \sum_{j=1}^N R_j(x) \quad \text{and} \quad \tilde R(x) = \sum_{j=1}^N \tilde R_j(x) . \]
 Let $ w= u-R$ small in $L^2$. Consider 
\[ \begin{array}{rrcl}
\Phi : & L^2 \times \m R^N &  \to & \m R^N, \\
& (w,\g y) & \mapsto & \displaystyle \left( \int (w + R - \tilde R) {\/\tilde R_j}_x \right)_{j=1,\ldots,N}.
\end{array} \]
 Let $\g z =(z_j)_{j=1,\ldots,N} $. By the decay properties of $\tilde R_j$,
\begin{align*}
& (d_{\g y} \Phi(w,\g y).\g z)_j =    \sum_{k=1}^N z_k \int {\/\tilde R_k}_x {\/\tilde R_j}_x 
- z_j \int (w + R - \tilde R) {\/\tilde R_j}_{xx} \\
& =  z_j \| {Q_{c_j}}_x \|_{L^2}^2 +  O \left( \sum_{k \ne j} e^{-\sigma_0 |\alpha_k-\alpha_j|} |z_k| \right)  
  + O(|z_j| \| w \|_{L^2}) + O(|z_j| \| \g y \|).
\end{align*}
Hence 
\begin{equation}\label{diag}
	d_{\g y} \Phi(w,\g y)= \text{diag}(\| {Q_{c_j}}_x \|_{L^2}^2)+O(\sum_{k \ne j} e^{-\sigma_0|\alpha_k-\alpha_j|})+ O(  \| w \|_{L^2}) + O( \| \g y \|).
\end{equation}
Therefore, if $\min \{ |\alpha_k-\alpha_j|, i \ne j \} $ is large enough then $d_{\g y} \Phi(0,0)$ is invertible. Since $\Phi(0,0)=0$, by the implicit function theorem, it follows that there exists $\epsilon >0$, $\epsilon\le \eta$ and a $C^1$ function $\phi : B_{L^2} (0,\epsilon) \to B_{\m R^N}(0,\eta)$ such that  $\Phi(w,\g y) =0$ in $B_{L^2} (0,\epsilon) \times \phi(B_{L^2} (0,\epsilon))$ is equivalent to $\g y=\phi(w)$. Finally we set $v=v(w)=w+ R - \sum_{j=1}^N R_j(\cdot - \phi(w)_j)$.
\end{proof}

\begin{proof}[Proof of Lemma \ref{finaldata}]
% Fix $T_0>0$ large enough and $\e >0$ small enough so that, from Lemma \ref{modulation}, for all $t \ge T_0$, all functions $u_0$ satifying $\| u_0 - R(t) \|_{H^1} < \e$ are modulable.
%
% Recall the notation $\gh b = (\gh b_j^\pm)_{j=1,\ldots,N;\pm}\in \m R^{2N}$. The function $u(S_n)$ being chosen as in \eqref{gkdvSn}, we have \[ \| u(S_n) - R(S_n) \|_{H^1} \le C_0 \| \gh b  \|. \] Now consider $n$ large enough so that $S_n \ge 2 T_0$ and $ C_0 e^{-\sigma_0^{3/2}S_n} \le \frac 12  \epsilon$.

Consider the maps~:
\begin{equation*}
\begin{array}{rrlrrlrrl}
 {\cal I}:  \   \m R^{2N}  & \to  &  H^1  \qquad    & \ \Theta: \  \q V & \to  & H^1 \times \m R^N \qquad  &  \ {\cal S}:  \  H^1 \times \m R^N & \to &  \m R^{2N} 
\\
 \gh b  & \mapsto &  \sum_{j,\pm} \gh b_j^\pm Z_j^\pm(S_n)   &   w  & \mapsto &  (v,\g y)  &    (v,\g y) & \mapsto &  \left( \int v {\tilde Z_j^\pm} \right)_{j,\pm}
\end{array}
\end{equation*}
where, in the definition of  $\Theta$,   $(v,\g y)$ represents  the modulation of $u = w+R(S_n)$ and $\q V=B_{H^1}(\epsilon)$ ($\epsilon$ being defined in the proof Lemma~\ref{modulation}), and  in the definition of ${\cal S}$, we have set $\tilde Z_j^\pm (x)= Z_j^\pm(S_n,x - y_j)$.

Then ${\cal I}(0) =0$, $\Theta(0)=(0,0)$ and ${\cal S}(0,0)=0$. Recall also from Lemma \ref{modulation}  that
\[ \| v \|_{L^2} + \| \g y\| + \| R_j(S_n) - \tilde R_j(S_n) \|_{H^1} \le C \|w \|_{L^2}. \]

To prove   Lemma \ref{finaldata}, we claim that $\Psi = {\cal S} \circ \Theta \circ {\cal I}$ is a diffeomorphism on a fixed neighbourhood of $0 \in \m R^{2N}$ by computing   $d\Psi = d{\cal S} \circ d\Theta \circ d{\cal I}$. Indeed, we claim

\begin{claim}
	\[ d\Psi(\gh b) = \left(
	\begin{array}{cc}
	A & (\int Z^+ Z^-) A \\
	(\int Z^+ Z^-) A & A
	\end{array} \right) + O(e^{-\sigma_0^{3/2} S_n}+\|\gh b\|), \]
	where $A = \mathrm{diag} ((\| Z_j \|_{L^2}^2)_j) =  \mathrm{diag} ((c_j^{\frac{5-p}{p-1}})_j)$ (recall that $\| Z^\pm \|_{L^2} =1$).
\end{claim}

\begin{nb}
	Note that if $N=1$ (only one soliton), with e.g. $c_1=1$, then the map $\Psi$ is represented by the matrix 
	\[ B= \left(\begin{array}{cc}
	 \int (Z^+)^2 & \int Z^+ Z^-  \\
	\int Z^+ Z^-  &  \int (Z^-)^2
	\end{array} \right)=\text{Gramm}(Z^\pm). \]
	Indeed, the functions $Z^{\pm}$ are orthogonal to $Q_x$, so that $\g y=0$ in this case and $\Psi$ is linear. Since $Z^{\pm}$ are linearly independent (see proof of Lemma \ref{ycoer}), the matrix $B$ is invertible.
	
	The claim means that for the general case $N\geq 2$, we obtain a similar behavior around each soliton plus small terms due to the interaction of the various solitons.
\end{nb}

\begin{proof}
We start with the computation of differentials of ${\cal I}$, $\Theta$ and ${\cal S}$. First, ${\cal I}$ is affine so that $d{\cal I}(\gh b)={\cal I}$ for all $\gh b$.
Second, for $h\in H^1$, $\g z\in \m R^{N}$,
\[ (dS(v,\g y).(h,\g z))_{j,\pm} =  z_j \int v {\/\tilde Z_j^\pm}_x   + \int h \tilde Z_j^\pm . \]
Finally, we consider  $\Theta$.  Let $\Phi$ and $\phi$ be defined as in the proof of the Lemma  \ref{modulation} above for $R(S_n)$.
Then, by \eqref{diag}, $d_{\g y} \Phi(w,\g y)$ is a diagonally dominant matrix and thus it is invertible. Denoting by $M$ its inverse, it follows from
\eqref{diag} that
\[ M = \mathrm{diag}( (\| {Q_{c_j}}_x \|_{L^2}^{-2})_j) + O(\| w \|_{L^2} + \|\g y \| + e^{-\sigma_0^{3/2} S_n}). \]
Differentiating $\Phi(w,\phi(w))=0$ with respect to $w$ and using $M=(d_{\g y} \Phi(w,\g y))^{-1}$, we find
$d\phi = - M  \circ d_w \Phi $.
Since $(d_w \Phi(w,\g y).h)_j = \int h {\/\tilde R_j}_x(S_n)$ and
$$\Theta(w) = \left(w + R-   \sum_j R_j(S_n, \cdot - \phi(w)_j), \phi(w)\right),$$
we obtain
\begin{align*}
 d\Theta(w).h & = (h - \sum_j ({\/\tilde R_j}_x(S_n) ((M \circ d_w \Phi).h)_j  , - M \circ d_w \Phi. h) \\
& = \left( h + \sum_{j=1}^N \| {Q_{c_j}}_x \|_{L^2}^{-2}  {\/\tilde R_j}_x(S_n)  \int h {\/\tilde R_j}_x(S_n) 
, \left( - \| {Q_{c_j}}_x \|_{L^2}^{-2} \int h {\/\tilde R_j}_x(S_n) ) 
\right)_{j=1,\ldots,N} \right)  \\
& \qquad \qquad + O(e^{-\sigma_0^{3/2} S_n} + \| w\|_{L^2}) \| h \|_{L^2}).
\end{align*}
 
Let $\tilde {\gh b}\in \m R^{2N}$. Then, since ${\cal I}$ is linear, we have
$
	d\Psi(\gh b). \tilde {\gh b} = dS(\Theta({\cal I}(\gh b))).(d\Theta({\cal I}(\gh b)).{\cal I}(\tilde {\gh b})).
$
By the previous computations, we have
\begin{equation*}\begin{split}
&d\Theta({\cal I}(\gh b)).{\cal I}(\tilde {\gh b})
\\ &= \left({\cal I}(\tilde {\gh b}) + \sum_{j=1}^N \| {Q_{c_j}}_x \|_{L^2}^{-2}  {\/\tilde R_j}_x(S_n)  \int {\cal I}(\tilde {\gh b}) {\/\tilde R_j}_x(S_n)
, \left( - \| {Q_{c_j}}_x \|_{L^2}^{-2} \int {\cal I}(\tilde {\gh b}) {\/\tilde R_j}_x(S_n) )
\right)_{j=1,\ldots,N} \right)
\end{split}\end{equation*}
Inserting the expression of ${\cal I}(\tilde {\gh b})$, using $\|\g y\|\leq C \|\gh b\|$, $\int Z^{\pm} Q_x =0$ and   the decay properties of the functions $Q$ and $Z$, we get
$$
	d\Theta({\cal I}(\gh b)).{\cal I}(\tilde {\gh b}) = ({\cal I}(\tilde {\gh b}),0) +  O(e^{-\sigma_0^{3/2} S_n}+ \|\gh b\|) \|\tilde {\gh b}\|.
$$
Therefore, using the expression of $d {\cal S}$, we finally obtain
\begin{equation*}
d \Psi(\gh b)  
= \mathrm{Gramm}((Z^\pm_j)_{j,\pm})  +  O( e^{-\sigma_0^{3/2} S_n}+ \|\gh b\|)  = P + O( e^{-\sigma_0^{3/2} S_n}+ \|\gh b\|)  
\end{equation*}
where $\mathrm{Gramm}((Z^\pm_j)_{j,\pm})$ is the Gramm matrix of the family $(Z^\pm_j)_{j,\pm}$
\[ \mathrm{Gramm}((Z^\pm_j)_{j,\pm})_{(j_1, \pm_1),(j_2,\pm_2)} = \int Z^{\pm_1}_{j_1} Z^{\pm_2}_{j_2}, \]
and
\[ P = \left(
\begin{array}{cc}
A & (\int Z^+ Z^-) A \\
(\int Z^+ Z^-) A & A
\end{array} \right), \]
where $A = \mathrm{diag} ((\| Z_j \|_{L^2}^2)_j) =  \mathrm{diag} ((c_j^{\frac{5-p}{p-1}})_j)$ (recall that $\| Z^\pm \|_{L^2} =1$).
This finishes the proof of the claim.
\end{proof} 
 
Since $P$ is invertible ($Z^+$ and $Z^-$ are independent, see proof of Lemma \ref{ycoer}), we deduce that $d \Psi$ is invertible on some ball $B_{\m R^{2N}}(\eta)$ ($\eta>0$ independent of $n$ for $n\geq n_0$ large enough). As a consequence, $\Psi$ is a diffeomorphism from $B_{\m R^{2N}}(\eta)$ to some neighbourhood $\q W$ of $0 \in \m R^{2N}$. Let $\delta >0$ be such that $B_{\m R^{2N}}(\delta) \subset \q W$. For any  $\gh a^+ \in B_{\m R^{N}}(\delta)$, there exist a unique $\gh b=\gh b(\gh a^+) \in B_{\m R^{2N}}(\eta)$ such that $\Psi(\gh b(\gh a^+)) = (\gh a^+, 0)$ and $\|\gh b(\gh a^+)\|\leq C \|\gh a^+\|$.
\end{proof}

\footnotesize

\addcontentsline{toc}{section}{References}
\bibliographystyle{amsplain}
\providecommand{\bysame}{\leavevmode\hbox to3em{\hrulefill}\thinspace}
\providecommand{\MR}{\relax\ifhmode\unskip\space\fi MR }
\providecommand{\MRhref}[2]{%
  \href{http://www.ams.org/mathscinet-getitem?mr=#1}{#2}
}
\providecommand{\href}[2]{#2}

\bigskip

\end{document}